\documentclass[a4paper,11pt,fleqn]{article}
\usepackage{pstricks}
\usepackage{amsmath}
\usepackage{amssymb}
\usepackage{theorem} 
\usepackage{euscript}
\usepackage{exscale,relsize}
\usepackage{graphicx}
\usepackage{srcltx}
\usepackage{xcolor}
\usepackage{url}
\usepackage{charter}
\newcommand{\email}[1]{\href{mailto:#1}{\nolinkurl{#1}}}
\renewcommand{\leq}{\ensuremath{\leqslant}}
\renewcommand{\geq}{\ensuremath{\geqslant}}

\newcommand{\minimize}[2]{\ensuremath{\underset{\substack{{#1}}}%
{\text{\rm minimize}}\;\;#2 }}

\newcommand{\Frac}[2]{\displaystyle{\frac{#1}{#2}}} 
 
\newcommand{\scal}[2]{{\left\langle{{#1}\mid{#2}}\right\rangle}}
\newcommand{\menge}[2]{\big\{{#1}~\big |~{#2}\big\}}

\newcommand{\HH}{\ensuremath{{\mathcal H}}}

\newcommand{\Sum}{\ensuremath{\displaystyle\sum}}
\newcommand{\emp}{\ensuremath{{\varnothing}}}

\newcommand{\Id}{\ensuremath{\mathrm{Id}}\,}

\newcommand{\RR}{\ensuremath{\mathbb{R}}}
\newcommand{\RP}{\ensuremath{\left[0,+\infty\right[}}

\newcommand{\RPP}{\ensuremath{\left]0,+\infty\right[}}

\newcommand{\RPX}{\ensuremath{\left[0,+\infty\right]}}
\newcommand{\RX}{\ensuremath{\left]-\infty,+\infty\right]}}

\newcommand{\NN}{\ensuremath{\mathbb N}}

\newcommand{\weakly}{\ensuremath{\:\rightharpoonup\:}}

\newcommand{\zer}{\ensuremath{\text{\rm zer}\,}}
\newcommand{\pinf}{\ensuremath{{+\infty}}}

\newcommand{\dom}{\ensuremath{\text{\rm dom}\,}}
\newcommand{\prox}{\ensuremath{\text{\rm prox}}}
\newcommand{\Fix}{\ensuremath{\text{\rm Fix}\,}}
\newcommand{\ii}{\ensuremath{\mathrm i}}

\newcommand{\gra}{\ensuremath{\text{\rm gra}}}
\newcommand{\inte}{\ensuremath{\text{\rm int}\,}}

\newcommand{\zeroun}{\ensuremath{\left]0,1\right[}}

\newtheorem{theorem}{Theorem}[section]
\newtheorem{lemma}[theorem]{Lemma}
\newtheorem{corollary}[theorem]{Corollary}
\newtheorem{proposition}[theorem]{Proposition}

\theoremstyle{plain}{\theorembodyfont{\rmfamily}%
}
\theoremstyle{plain}{\theorembodyfont{\rmfamily}%
}
\theoremstyle{plain}{\theorembodyfont{\rmfamily}%
\newtheorem{remark}[theorem]{Remark}}
\theoremstyle{plain}{\theorembodyfont{\rmfamily}%
}
\theoremstyle{plain}{\theorembodyfont{\rmfamily}%
}
\theoremstyle{plain}{\theorembodyfont{\rmfamily}%
\newtheorem{definition}[theorem]{Definition}}
\theoremstyle{plain}{\theorembodyfont{\rmfamily}%
}

\numberwithin{equation}{section}
\begin{document}

\title{\sffamily\huge Compositions and Convex Combinations of\\ 
Averaged Nonexpansive Operators\footnote{Contact 
author: P. L. Combettes, {\ttfamily plc@ljll.math.upmc.fr},
phone: +33 1 4427 6319, fax: +33 1 4427 7200.}}

\author{Patrick L. Combettes$^1$ and Isao Yamada$^2$\\[5mm]
\small $\!^1$Sorbonne Universit\'es -- UPMC Univ. Paris 06\\
\small UMR 7598, Laboratoire Jacques-Louis Lions\\
\small F-75005 Paris, France\\
\small \ttfamily{plc@ljll.math.upmc.fr}\\[4mm]
\small $\!^2$Tokyo Institute of Technology\\
\small Department of Communications and Computer Engineering\\
\small Tokyo 152-8550, Japan\\
\small \ttfamily{isao@sp.ce.titech.ac.jp}\\[4mm]
}

\date{~}

\maketitle

\vskip 8mm

\begin{abstract} \noindent
Properties of compositions and convex combinations of averaged 
nonexpansive operators are investigated and applied to the
design of new fixed point algorithms in Hilbert spaces. An extended
version of the forward-backward splitting algorithm for finding a
zero of the sum of two monotone operators is obtained.
\end{abstract} 


\newpage

\section{Introduction}

Since their introduction in \cite{Bail78}, averaged nonexpansive 
operators have proved to be very useful in the analysis and the
numerical solution of problems arising in nonlinear analysis and 
its applications; see, e.g., \cite{Bail92,Bail14,Baus96,Livre1,%
Byrn04,Cegi12,Opti04,Cond13,Crom07,Ogur02,Ragu13,Weny08,%
Yama01,Yama04}.

\begin{definition}
\label{d:averaged}
Let $\HH$ be a real Hilbert space, let $D$ be a nonempty subset of 
$\HH$, let $\alpha\in\zeroun$, and let $T\colon D\to\HH$ be a
nonexpansive (i.e., 1-Lipschitz) operator.
Then $T$ is averaged with constant $\alpha$, or 
$\alpha$-averaged, if there exists a nonexpansive operator 
$R\colon D\to\HH$ such that $T=(1-\alpha)\Id+\alpha R$.
\end{definition}

As discussed in \cite{Livre1,Opti04,Ogur02}, averaged operators 
are stable under compositions and convex combinations and such
operations form basic building blocks in various composite fixed 
point algorithms. The averagedness constants resulting from such
operations determine the range of the step sizes and other 
parameters in such algorithms. It is therefore important that they 
be tight since these parameters have a significant impact on the 
speed of convergence. 

In this paper, we discuss averagedness constants for compositions 
and convex combinations of averaged operators and construct novel 
fixed point algorithms based on these constants. 
In particular, we obtain a new version of the forward-backward 
algorithm with an extended relaxation range and iteration-dependent
step sizes.

Throughout the paper, $\HH$ is a real Hilbert space with scalar
product $\scal{\cdot}{\cdot}$ and associated norm $\|\cdot\|$.
We denote by $\Id$ the identity operator on $\HH$ and by $d_S$ the
distance function to a set $S\subset\HH$;
$\weakly$ and $\to$ denote, respectively, weak and strong 
convergence in $\HH$.

\section{Compositions and convex combinations of averaged operators}
\label{sec:2}

We first recall some characterizations of averaged operators (see 
\cite[Lemma~2.1]{Opti04} or \cite[Proposition~4.25]{Livre1}).

\begin{proposition}
\label{p:av1}
Let $D$ be a nonempty subset of $\HH$, let $T\colon D\to\HH$ be 
nonexpansive, and let $\alpha\in\zeroun$. 
Then the following are equivalent:
\begin{enumerate}
\item
\label{p:av1i}
$T$ is $\alpha$-averaged.
\item
\label{p:av1i'}
$(1-1/\alpha)\Id+(1/\alpha)T$ is nonexpansive.
\item
\label{p:av1ii}
$(\forall x\in D)(\forall y\in D)$ 
$\|Tx-Ty\|^2\leq\|x-y\|^2-\displaystyle{\frac{1-\alpha}{\alpha}}
\|(\Id-T)x-(\Id-T)y\|^2$.
\item
\label{p:av1iii}
$(\forall x\in D)(\forall y\in D)$ 
$\|Tx-Ty\|^2+(1-2\alpha)\|x-y\|^2\leq 2(1-\alpha)\scal{x-y}{Tx-Ty}$.
\end{enumerate}
\end{proposition}

The next result concerns the averagedness of a convex combination 
of averaged operators.

\begin{proposition}
\label{p:av2}
Let $D$ be a nonempty subset of $\HH$, let $(T_i)_{i\in I}$ be a 
finite family of nonexpansive operators from $D$ to $\HH$, let 
$(\alpha_i)_{i\in I}$ be a family in $\zeroun$, and let 
$(\omega_i)_{i\in I}$ be a family in $]0,1]$ such that 
$\sum_{i\in I}\omega_i=1$. Suppose that, for every $i\in I$, 
$T_i$ is $\alpha_i$-averaged, and set 
$T=\sum_{i\in I}\omega_i T_i$ and
$\alpha=\sum_{i\in I}\omega_i\alpha_i$. Then
$T$ is $\alpha$-averaged. 
\end{proposition}
\begin{proof}
For every $i\in I$, there exists a nonexpansive operator
$R_i\colon D\to\HH$ such that
$T_i=(1-\alpha_i)\Id+\alpha_i R_i$.
Now set $R=(1/\alpha)\sum_{i\in I}\omega_i\alpha_i R_i$. Then $R$ 
is nonexpansive and
\begin{equation}
\sum_{i\in I}\omega_iT_i=\sum_{i\in I}\omega_i(1-\alpha_i)\Id
+\sum_{i\in I}\omega_i\alpha_i R_i=(1-\alpha)\Id+\alpha R.
\end{equation}
We conclude that $T$ is $\alpha$-averaged.
\end{proof}

\begin{remark}
\label{raQg507}
In view of \cite[Corollary~2.2.17]{Cegi12}, Proposition~\ref{p:av2} 
is equivalent to \cite[Theorem~2.2.35]{Cegi12}, and it improves the 
averagedness constant of \cite[Lemma~2.2(ii)]{Opti04} which was 
$\alpha=\text{\rm max}_{i\in I}\alpha_i$. In 
the case of two operators, Proposition~\ref{p:av2} can be found in 
\cite[Theorem~3(a)]{Ogur02}. 
\end{remark}

Next, we turn our attention to compositions of averaged operators,
starting with the following result, which was obtained in 
\cite[Theorem~3(b)]{Ogur02} with a different proof.

\begin{proposition}
\label{paQg507-03}
Let $D$ be a nonempty subset of $\HH$, let 
$(\alpha_1,\alpha_2)\in\zeroun^2$, let $T_1\colon D\to D$ be 
$\alpha_1$-averaged, and let $T_2\colon D\to D$ be 
$\alpha_2$-averaged. Set
\begin{equation}
\label{XAre2e26a}
T=T_1 T_2\quad\text{and}\quad
\alpha=\frac{\alpha_1+\alpha_2-2\alpha_1\alpha_2}
{1-\alpha_1\alpha_2}.
\end{equation}
Then $\alpha\in\zeroun$ and  $T$ is $\alpha$-averaged.
\end{proposition}
\begin{proof}
Since $\alpha_1(1-\alpha_2)<(1-\alpha_2)$, we have
$\alpha_1+\alpha_2<1+\alpha_1\alpha_2$ and, therefore, 
$\alpha\in\zeroun$. Now let $x\in D$, let $y\in D$, and set
\begin{equation}
\label{Gt^reeb8718g}
\tau=\frac{1-\alpha_1}{\alpha_1}+
\frac{1-\alpha_2}{\alpha_2}. 
\end{equation}
It follows from Proposition~\ref{p:av1} that
\begin{align}
\label{e:TfgtF5e9-14a}
\|T_1 T_2x-T_1 T_2 y\|^2
&\leq\|T_2 x-T_2 y\|^2-\frac{1-\alpha_1}{\alpha_1}
\left\|(\Id-T_1)T_2 x-(\Id-T_1)T_2 y\right\|^2\nonumber\\
&\leq \|x-y\|^2-\frac{1-\alpha_2}{\alpha_2}
\left\|(\Id-T_2)x-(\Id-T_2)y\right\|^2\nonumber \\
&\quad\;-\frac{1-\alpha_1}{\alpha_1}
\left\|(\Id-T_1)T_2 x-(\Id-T_1)T_2 y\right\|^2.
\end{align}
Moreover, by \cite[Corollary~2.14]{Livre1}, we have
\begin{align}
\label{e:TfgtF5e9-14b}
&\hskip -8 mm 
\frac{1-\alpha_1}{\tau\alpha_1}\left\|(\Id-T_1)T_2
x-(\Id-T_1)T_2 y\right\|^2+\frac{1-\alpha_2}{\tau\alpha_2}
\left\|(\Id-T_2)x-(\Id-T_2)y\right\|^2\nonumber\\
&=\left\|\frac{1-\alpha_1}{\tau\alpha_1}
\big((\Id-T_1)T_2 x-(\Id-T_1)T_2 y\big)+
\frac{1-\alpha_2}{\tau\alpha_2}\big((\Id-T_2)x-(\Id-T_2)y\big)
\right\|^2\nonumber\\
&\quad\;+\frac{(1-\alpha_1)(1-\alpha_2)}{\tau^2\alpha_1\alpha_2}
\left\|(x-y)-\left(T_{1}T_2x-T_{1}T_2y\right)
\right\|^2\nonumber \\
&\geq\frac{(1-\alpha_1)(1-\alpha_2)}{\tau^2\alpha_1\alpha_2}
\left\|(\Id-T_1T_2)x-(\Id-T_1T_2)y\right\|^2. 
\end{align}
Combining \eqref{e:TfgtF5e9-14a}, \eqref{e:TfgtF5e9-14b}, 
and \eqref{XAre2e26a} yields
\begin{align}
\|T_1 T_2 x-T_1 T_2 y\|^2
&\leq\|x-y\|^2-\frac{(1-\alpha_1)(1-\alpha_2)}
{\tau\alpha_1\alpha_2}
\left\|\left(\Id-T_1T_2\right)x-\left(\Id-T_1T_2\right)y\right\|^2
\nonumber\\
&=\|x-y\|^2-\frac{1-\alpha_1-\alpha_2+\alpha_1\alpha_2}
{\alpha_1+\alpha_2-2\alpha_1\alpha_2}
\left\|\left(\Id-T_1T_2\right)x-\left(\Id-T_1T_2\right)y\right\|^2
\nonumber\\
&=\|x-y\|^2-\frac{1-\alpha}{\alpha}
\left\|\left(\Id-T_1T_2\right)x-\left(\Id-T_1T_2\right)y\right\|^2.
\end{align} 
In view of Proposition~\ref{p:av1}, we conclude that $T$ is
$\alpha$-averaged.
\end{proof}

In \cite[Theorem~2.2.37]{Cegi12}, the averagedness constant of 
\eqref{XAre2e26a} was written as 
\begin{equation}
\label{XAre2e26b}
\alpha=\dfrac{1}{1+\dfrac{1}{
\dfrac{\alpha_1}{1-\alpha_1}+\dfrac{\alpha_2}{1-\alpha_2}}}.
\end{equation}
By induction, it leads to the following result for the
composition of $m$ averaged operators, which was obtained
in \cite{Cegi12} (combine 
\cite[Theorem~2.2.42]{Cegi12} and \cite[Corollary~2.2.17]{Cegi12}).

\begin{proposition}
\label{paQg507-03'}
Let $D$ be a nonempty subset of $\HH$, let $m\geq 2$ be an integer, 
and set 
\begin{equation}
\label{hhghwedw710f27}
\phi\colon\zeroun^m\to\zeroun\colon
(\alpha_1,\ldots,\alpha_m)\mapsto\dfrac{1}{1+\dfrac{1}{\Sum_{i=1}^m
\dfrac{\alpha_i}{1-\alpha_i}}}.
\end{equation}
For every $i\in\{1,\ldots,m\}$, let 
$\alpha_i\in\zeroun$ and let $T_i\colon D\to D$ be 
$\alpha_i$-averaged. Set 
\begin{equation}
\label{XAre2e25}
T=T_1\cdots T_m\quad\text{and}\quad
\alpha=\phi(\alpha_1,\ldots,\alpha_m).
\end{equation}
Then $T$ is $\alpha$-averaged.
\end{proposition}
\begin{proof}
We proceed by induction on $k\in\{2,\ldots,m\}$. To this end, let
us set $(\forall k\in\{2,\ldots,m\})$
$\beta_k=[1+[\sum_{i=1}^k\alpha_i/(1-\alpha_i)]^{-1}]^{-1}$.
By Proposition~\ref{paQg507-03} and \eqref{XAre2e26b},
the claim is true for $k=2$.
Now assume that, for some $k\in\{2,\ldots,m-1\}$,
$T_1\cdots T_{k}$ is $\beta_k$-averaged. Then we deduce from 
Proposition~\ref{paQg507-03} and \eqref{XAre2e26b} 
that the averagedness constant of $(T_1\cdots T_{k})T_{k+1}$ is 
\begin{equation}
\label{XAre2e26c}
\dfrac{1}{1+\dfrac{1}{
\dfrac{1}{\beta_k^{-1}-1}+\dfrac{\alpha_{k+1}}{1-\alpha_{k+1}}}}
=\dfrac{1}{1+\dfrac{1}{\bigg(\Sum_{i=1}^k\dfrac{\alpha_i}
{1-\alpha_i}
\bigg)+\dfrac{\alpha_{k+1}}{1-\alpha_{k+1}}}}=\beta_{k+1},
\end{equation}
which concludes the induction argument.
\end{proof}

The following result provides alternative expressions for the
averagedness constant $\alpha$ of \eqref{XAre2e25}.

\begin{proposition}
\label{paQg507-03''}
Let $m\geq 2$ be an integer, let $\phi$ be as in 
\eqref{hhghwedw710f27}, 
let $(\alpha_i)_{1\leq i\leq m}\in\zeroun^m$,
and let $(\sigma_j)_{1\leq j\leq m}$ the elementary symmetric 
polynomials in the variables $(\alpha_i)_{1\leq i\leq m}$, i.e.,
\begin{equation}
\label{Bv54Ewe309-27b}
(\forall j\in\{1,\ldots,m\})\quad
\sigma_j=\sum_{1\leq i_1<\cdots<i_j\leq m}
\prod_{l=1}^j\alpha_{i_l}. 
\end{equation}
Then the following hold:
\begin{enumerate}
\item 
\label{paQg507-03''ii} 
$\phi(\alpha_1,\ldots,\alpha_m)=\big[
{\sum_{l=1}^{\pinf}\sum_{i=1}^{m}\alpha_{i}^{l}}\big]\big/\big[
{1+\sum_{l=1}^{\pinf}\sum_{i=1}^{m}\alpha_{i}^{l}}\big]$.
\item
\label{paQg507-03''iii}
$\phi(\alpha_1,\ldots,\alpha_m)=
\big[{\sum_{j=1}^m(-1)^{j-1}j\sigma_j}\big]\big/
\big[{1+\sum_{j=2}^{m} (-1)^{j-1}(j-1)\sigma_j}\big]$.
\item 
\label{paQg507-03''i} 
$\phi(\alpha_1,\ldots,\alpha_m)>\text{\rm max}_{1\leq i\leq m}
\alpha_i$.
\end{enumerate}
\end{proposition}
\begin{proof}
\ref{paQg507-03''ii}:
Indeed, \eqref{hhghwedw710f27} yields
\begin{equation}
\phi(\alpha_1,\ldots,\alpha_m)=
\frac{\Sum_{i=1}^{m}\left(\frac{1}{1-\alpha_i}-1\right)}
{1+\Sum_{i=1}^{m}\left(\frac{1}{1-\alpha_i}-1\right)},\quad
\text{where}\quad(\forall i\in\{1,\ldots,m\})\quad
\frac{1}{1-\alpha_i}-1=\sum_{l=1}^{\pinf}\alpha_i^{l}. 
\end{equation}

\ref{paQg507-03''iii}:
Using the inductive argument of the proof of Proposition
\ref{paQg507-03'} and \eqref{XAre2e26b}, we observe
that $\phi(\alpha_1,\ldots,\alpha_m)$ can be defined 
via the recursion
\begin{equation}
\label{Gt^reeb8718f}
\begin{array}{l}
\left\lfloor
\begin{array}{l}
\beta_1=\alpha_1\\
\text{for}\;k=1,\ldots,m-1\\
\begin{array}{l}
\left\lfloor
\begin{array}{l}
\beta_{k+1}=\Frac{\alpha_{k+1}+\beta_k-2\alpha_{k+1}\beta_k}
{1-\alpha_{k+1}\beta_k}
\end{array}
\right.\\[2mm]
\end{array}\\
\phi(\alpha_1,\ldots,\alpha_m)=\beta_m.
\end{array}
\right.
\end{array}
\end{equation} 
Set
\begin{equation}
\label{e:BnBn6Tf3419b}
(\forall k\in\{1,\ldots,m \})\quad
\begin{cases}
s_{0}(k)=1\\
s_{k+1}(k)=0\\
(\forall j\in\{1,\ldots, k\})\quad
s_{j}(k)=\sum_{1\leq i_1<\cdots<i_j\leq k}
\prod_{l=1}^j\alpha_{i_l}.
\end{cases}
\end{equation} 
We have $(\forall j\in\{1,\ldots, m\})$ $\sigma_j=s_j(m)$.
Furthermore, 
\begin{equation}
\label{e:BnBn6Tf3418b}
(\forall k\in\{1,\ldots,m-1\})(\forall j\in\{0,\ldots,k\})\quad
s_{j+1}(k+1)=s_{j+1}(k)+\alpha_{k+1}s_{j}(k).
\end{equation} 
Let us show by induction that, for every $k\in \{2,\ldots, m\}$,  
\begin{equation}
\label{e:n6Tk13420a}
\beta_{k}=\frac{\sum_{j=1}^k(-1)^{j-1}js_j(k)}{1+\sum_{j=2}^{k} 
(-1)^{j-1}(j-1)s_{j}(k)}.
\end{equation} 
Since $s_{1}(2)=\alpha_1+\alpha_2$ and
$s_2(2)=\alpha_1\alpha_2$, \eqref{Gt^reeb8718f} yields
\begin{equation}
\beta_2=\frac{\alpha_2+\beta_1-2\alpha_2\beta_1}{1-\alpha_2 \beta_1}
=\frac{\alpha_2+\alpha_1-2\alpha_2\alpha_1}{1-\alpha_2\alpha_1}
=\frac{s_{1}(2)-2s_2(2)}{1-s_2(2)}. 
\end{equation}
This establishes \eqref{e:n6Tk13420a} for $k=2$. 
Now suppose that \eqref{e:n6Tk13420a} holds for some 
$k\in \{2,\ldots, m-1\}$. We derive from 
\eqref{e:BnBn6Tf3419b} and \eqref{e:BnBn6Tf3418b} that
\begin{align}
\label{e:BnBn6Tf3418d}
&\hskip -6mm
\big(\alpha_{k+1}+\beta_k-2\alpha_{k+1}\beta_k\big)
\Bigg(1+\sum_{j=2}^{k}(-1)^{j-1}(j-1)s_{j}(k)\Bigg)\nonumber\\
&=\alpha_{k+1}\bigg(1+\sum_{j=2}^{k}(-1)^{j-1}(j-1)s_{j}(k)\bigg)
\nonumber\\
&\quad\;+\sum_{j=1}^{k}(-1)^{j-1}js_{j}(k)-2\alpha_{k+1}
\bigg(s_{1}(k)+\sum_{j=2}^{k}(-1)^{j-1}js_{j}(k)\bigg)\nonumber\\
&=\alpha_{k+1}+s_{1}(k)+
\sum_{j=2}^{k}(-1)^{j-1}js_{j}(k)\nonumber\\
&\quad\;+\alpha_{k+1}\bigg(\sum_{j=2}^{k}(-1)^{j-1}(j-1)s_{j}(k)
-2\sum_{j=2}^{k}(-1)^{j-1}js_{j}(k)-2s_{1}(k)\bigg)
\nonumber\\
&=s_{1}(k+1)+\sum_{j=2}^{k}(-1)^{j-1}js_{j}(k)
-\alpha_{k+1}\bigg(\sum_{j=2}^{k}(-1)^{j-1}(j+1)s_{j}(k)
+2s_{1}(k)\bigg)\nonumber\\
&=s_{1}(k+1)+\sum_{j=2}^{k}(-1)^{j-1}js_{j}(k)
-\alpha_{k+1}\sum_{j=1}^{k}(-1)^{j-1}(j+1)s_{j}(k)\nonumber\\
&=s_{1}(k+1)+\sum_{j=2}^{k}(-1)^{j-1}js_{j}(k)
-\alpha_{k+1}\sum_{j=1}^{k-1}(-1)^{j-1}(j+1)s_{j}(k)\nonumber\\
&\quad\; -(-1)^{k-1}(k+1)\alpha_{k+1}s_{k}(k)\nonumber\\
&=s_{1}(k+1)+\sum_{j=2}^{k}(-1)^{j-1}js_{j}(k)
+\alpha_{k+1}\sum_{j=2}^{k}(-1)^{j-1}js_{j-1}(k)
+(-1)^{k}(k+1)s_{k+1}(k+1)\nonumber\\
&=s_{1}(k+1)+\sum_{j=2}^{k}(-1)^{j-1}j\big(s_{j}(k)
+\alpha_{k+1}s_{j-1}(k)\big)
+(-1)^{k}(k+1)s_{k+1}(k+1)\nonumber\\
&=s_{1}(k+1)+\sum_{j=2}^{k}(-1)^{j-1}js_{j}(k+1)
+(-1)^{k}(k+1)s_{k+1}(k+1)\nonumber\\
&=\sum_{j=1}^{k+1}(-1)^{j-1}js_{j}(k+1)
\end{align}
and that
\begin{align}
\label{e:BnBn6Tf3418c}
&\hskip -7mm
\big(1-\alpha_{k+1}\beta_k\big)
\Bigg(1+\sum_{j=2}^{k}(-1)^{j-1}(j-1)s_{j}(k)\Bigg)\nonumber\\
&=\Bigg(1-\frac{\sum_{j=1}^{k}(-1)^{j-1}js_{j}(k)\alpha_{k+1}}
{1+\sum_{j=2}^{k}(-1)^{j-1}(j-1)s_{j}(k)}\Bigg)\Bigg(
1+\sum_{j=2}^{k}(-1)^{j-1}(j-1)s_{j}(k)\Bigg)\nonumber \\
&=1+\sum_{j=2}^{k}(-1)^{j-1}(j-1)s_{j}(k)-\sum_{j=1}^{k}(-1)^{j-1}j
s_{j}(k)\alpha_{k+1}\nonumber \\
&=1+\sum_{j=2}^{k}(-1)^{j-1}(j-1)s_{j}(k)+\sum_{j=1}^{k}(-1)^{j-1}j
\big(s_{j+1}(k)-s_{j+1}(k+1)\big)\nonumber\\
&=1+\sum_{j=2}^{k}(-1)^{j-1}(j-1)s_{j}(k)-
\sum_{j=2}^{k+1}(-1)^{j-1}(j-1)s_{j}(k)+\sum_{j=2}^{k+1}
(-1)^{j-1}(j-1)s_{j}(k+1)\nonumber\\
&=1+\sum_{j=2}^{k+1}(-1)^{j-1}(j-1)s_{j}(k+1). 
\end{align}
Taking the ratio of \eqref{e:BnBn6Tf3418d} and 
\eqref{e:BnBn6Tf3418c} yields
\begin{equation}
\label{ewope36}
\beta_{k+1}=\Frac{\alpha_{k+1}+\beta_k-2\alpha_{k+1}\beta_k}
{1-\alpha_{k+1}\beta_k}=\frac{\sum_{j=1}^{k+1}(-1)^{j-1}js_j(k+1)}
{1+\sum_{j=2}^{k+1}(-1)^{j-1}(j-1)s_{j}(k+1)}.
\end{equation}
This shows that \eqref{e:n6Tk13420a} holds 
for every $k\in \{2,\dots,m\}$. 

\ref{paQg507-03''i}: 
We need to consider only the case when $m=2$ since the general case
will follow from \eqref{Gt^reeb8718f} by induction. 
We derive from \eqref{Gt^reeb8718f} that
\begin{equation}
\label{e:m=2}
\beta_2=\frac{\alpha_1+\alpha_2-2\alpha_1\alpha_2}
{1-\alpha_1\alpha_2}. 
\end{equation}
Since 
$\beta_2-\alpha_1=\alpha_2(1-\alpha_1)^{2}/(1-\alpha_1\alpha_2)>0$ 
and
$\beta_2-\alpha_2=\alpha_1(1-\alpha_2)^{2}/(1-\alpha_1\alpha_2)>0$,
we have $\beta_2>\text{max}\{\alpha_1,\alpha_2\}>0$.  
\end{proof}

\begin{remark}
\label{KWghy5-f-09}
Let us compare the averagedness constant of 
Proposition~\ref{paQg507-03'} with alternative ones. Set 
\begin{equation}
\label{e:opti2004}
\widetilde{\phi}\colon\zeroun^m\to\zeroun\colon
(\alpha_1,\ldots,\alpha_m)\mapsto
\Frac{m\,\text{max}\{\alpha_1,\ldots,\alpha_m\}} 
{(m-1)\text{max}\{\alpha_1,\ldots,\alpha_m\}+1},
\end{equation}
and let $(\alpha_i)_{1\leq i\leq m}\in\zeroun^m$.
\begin{enumerate}
\item 
\label{KWghy5-f-09i}
The averagedness constant of Proposition~\ref{paQg507-03'} 
is sharper than that of \cite[Lemma~2.2(iii)]{Opti04}, namely 
\begin{equation}
\label{e:opti04}
\phi(\alpha_1,\ldots,\alpha_m)\leq 
\widetilde{\phi}(\alpha_1,\ldots,\alpha_m).
\end{equation}
\item 
\label{r:iuiiipii}
${\phi}(\alpha_1,\ldots,\alpha_m)=
\widetilde{\phi}(\alpha_1,\ldots,\alpha_m)$ if
$\alpha_1=\cdots=\alpha_m$ and, in particular, if all the 
operators are firmly nonexpansive, i.e., 
$\alpha_1=\cdots=\alpha_m=1/2$.
\item 
\label{KWghy5-f-09iii}
If $m=2$, the averagedness constant of 
Proposition~\ref{paQg507-03'} is strictly sharper than 
that of \cite[Lemma~3.2]{Weny08}, namely (see also 
\cite[Remark~2.2.38]{Cegi12})
\begin{equation}
\label{e:chin08}
\phi(\alpha_1,\alpha_2)<\widehat{\phi}(\alpha_1,\alpha_2),
\quad\text{where}\quad
\widehat{\phi}(\alpha_1,\alpha_2)=\alpha_1+\alpha_2-\alpha_1\alpha_2.
\end{equation}
In addition, $\phi(\alpha_1,\alpha_1)=
\widetilde{\phi}(\alpha_1,\alpha_1)<
\widehat{\phi}(\alpha_1,\alpha_1)$ while, for 
$\alpha_1={3}/{4}$ and $\alpha_2={1}/{8}$, 
$\widehat{\phi}(\alpha_1,\alpha_2)={25}/{32}<
{6}/{7}=\widetilde{\phi}(\alpha_1,\alpha_2)$, which shows that
$\widetilde{\phi}$ and $\widehat{\phi}$ cannot be compared in
general.
\end{enumerate}
\end{remark}
\begin{proof}
\ref{KWghy5-f-09i}: 
Combine \cite[Theorem~2.2.42]{Cegi12}, 
and \cite[Corollary~2.2.17]{Cegi12}.

\ref{r:iuiiipii}: 
Set $\beta_1=\delta_1=\alpha_1$ and
\begin{equation}
\label{e:BnBn6Tf3421b}
(\forall k\in\{2,\ldots,m\})\quad
\begin{cases}
\beta_k=\dfrac{1}{1+\dfrac{1}{\Sum_{i=1}^k
\dfrac{\alpha_i}{1-\alpha_i}}},\\[18mm]
\delta_k=\Frac{k\,\text{max}\{\alpha_1,\ldots,\alpha_k\}} 
{(k-1)\text{max}\{\alpha_1,\ldots,\alpha_k\}+1}.
\end{cases}
\end{equation} 
Then \eqref{e:BnBn6Tf3421b} yields
\begin{equation}
\label{Bv54Ewe36-9a}
(\forall k\in\{1,\ldots,m\})\quad
\delta_k=\frac{k\alpha_1}{(k-1)\alpha_1+1}. 
\end{equation} 
Let us show by induction that 
\begin{equation}
\label{jIv5989e36-9b}
(\forall k\in\{1,\ldots,m\})\quad \beta_k=\delta_k.
\end{equation} 
We have $\beta_1=\delta_1=\alpha_1$. Next, suppose that, for some 
$k\in\{1,\ldots,m-1\}$, $\beta_k=\delta_k$. 
Then $\alpha_{k+1}=\alpha_1$, while
\eqref{XAre2e26c} and \eqref{Bv54Ewe36-9a} yield 
\begin{equation}
\label{e:case4} 
\beta_{k+1}=\dfrac{1}{1+\dfrac{1}{
\dfrac{1}{\beta_k^{-1}-1}+\dfrac{\alpha_1}{1-\alpha_1}}}
=\dfrac{1}{1+\dfrac{1}{
\dfrac{1}{\delta_k^{-1}-1}+\dfrac{\alpha_1}{1-\alpha_1}}}
=\frac{(k+1)\alpha_1}{k\alpha_1+1}=\delta_{k+1}. 
\end{equation}
This establishes \eqref{jIv5989e36-9b}.

\ref{KWghy5-f-09iii}: 
This inequality was already obtained in 
\cite[Remark~2.2.38]{Cegi12}. It follows from the fact that
\begin{equation}
\widehat{\phi}(\alpha_1,\alpha_2)-
\phi(\alpha_1,\alpha_2)=\frac{\alpha_1\alpha_2 
(1-\alpha_1)(1-\alpha_2)}{1-\alpha_1\alpha_2}>0.
\end{equation}
The remaining assertions are easily verified.
\end{proof}

\section{Algorithms}
\label{sec:3}

We present applications of the bounds discussed in
Section~\ref{sec:2} to fixed point algorithms. Henceforth, we 
denote the set of fixed points of an operator $T\colon\HH\to\HH$ by
$\Fix T$.

As a direct application of Proposition~\ref{p:av2} and 
Proposition~\ref{paQg507-03'}, we first consider so-called 
``string-averaging'' iterations, which involve a mix of 
compositions and convex combinations of operators. In the case of
projection operators, such iterations go back to \cite{Cens01}.

\begin{proposition}
\label{p:nZZy418}
Let $(T_i)_{i\in I}$ be a finite family of nonexpansive operators 
from $\HH$ to $\HH$ such that $\bigcap_{i\in I}\Fix T_i\neq\emp$, 
and let $(\alpha_i)_{i\in I}$ be real numbers in 
$\left]0,1\right[$ such that, for every $i\in I$, $T_i$ is 
$\alpha_i$-averaged. Let $p$ be a strictly positive integer, 
for every $k\in\{1,\ldots,p\}$ let $m_k$ be a strictly positive 
integer and let $\omega_k\in\left]0,1\right]$, and suppose that 
$\ii\colon\menge{(k,l)}{k\in\{1,\ldots,p\},\,l\in
\{1,\ldots,m_k\}}\to I$
is surjective and that $\sum_{k=1}^{p}\omega_k=1$. Define
\begin{equation}
\label{e:nZZy418}
T=\sum_{k=1}^{p}\omega_kT_{\ii(k,1)}\cdots T_{\ii(k,m_k)}.
\end{equation}
Then the following hold: 
\begin{enumerate}
\item 
\label{p:nZZy418i} 
Set
\begin{equation}
\label{e:nZZy418b}
\alpha=\sum_{k=1}^{p}\dfrac{\omega_k}{1+\dfrac{1}
{\Sum_{i=1}^{m_k}\dfrac{\alpha_{\ii(k,i)}}{1-\alpha_{\ii(k,i)}}}}
\end{equation}
Then $T$ is $\alpha$-averaged and 
$\Fix T=\bigcap_{i\in I}\Fix T_i$. 
\item 
\label{p:nZZy418ii}
Let $(\lambda_n)_{n\in \NN}$ be a sequence in
$\left]0,1/\alpha\right[$ such that 
$\sum_{n\in\NN}\lambda_n(1/\alpha-\lambda_n)=\pinf$.
Furthermore, let $x_0\in \HH$ and set 
\begin{equation}
(\forall n\in \NN)\quad
x_{n+1}=x_n+\lambda_n\big(Tx_n-x_n\big). 
\end{equation}
Then $(x_n)_{n\in \NN}$ converges weakly to a point in
$\bigcap_{i\in I}\Fix T_i$.
\end{enumerate}
\end{proposition}
\begin{proof}
\ref{p:nZZy418i}:
The $\alpha$-averagedness of $T$ follows from 
Propositions~\ref{p:av2} and ~\ref{paQg507-03'}. 
The remaining assertions follow from 
\cite[Proposition~4.34 and Corollary~4.37]{Livre1}. 
 
\ref{p:nZZy418ii}: This follows from 
\ref{p:nZZy418i} and 
\cite[Proposition~5.15(iii)]{Livre1}.
\end{proof}

\begin{remark}
Proposition~\ref{p:nZZy418} improves upon 
\cite[Corollary~5.18]{Livre1}, where the averagedness constant
$\alpha$ of \eqref{e:nZZy418b} was replaced by 
\begin{equation}
\label{e:2009-03-21a}
\alpha'=\underset{1\leq k\leq p}{\text{\rm max}}\:
\rho_k,\quad\text{with}\quad
(\forall k\in\{1,\ldots,p\})\quad
\rho_k=\frac{m_k}{m_k-1+\displaystyle{\frac{1}{\text{\rm max}
\big\{\alpha_{\ii(k,1)},\ldots,\alpha_{\ii(k,m_k)}\big\}}}}.
\end{equation}
In view of Remarks~\ref{raQg507} and 
\ref{KWghy5-f-09}\ref{KWghy5-f-09i},
$\alpha'\geq\alpha$ and therefore $\alpha$ provides a larger
range for the relaxation parameters $(\lambda_n)_{n\in\NN}$.
\end{remark}

The subsequent applications require the following technical fact.

\begin{lemma}{\rm\cite[Lemma~2.2.2]{Poly87}}
\label{l:7}
Let $(\alpha_n)_{n\in\NN}$, $(\beta_n)_{n\in\NN}$, and 
$(\varepsilon_n)_{n\in\NN}$ be sequences in $\RP$ 
such that $\sum_{n\in\NN}\varepsilon_n<\pinf$ and 
$(\forall n\in\NN)$ $\alpha_{n+1}\leq\alpha_n-\beta_n
+\varepsilon_n$. Then $(\alpha_n)_{n\in\NN}$ converges and 
$\sum_{n\in\NN}\beta_n<\pinf$.
\end{lemma}

Next, we introduce a general iteration process for finding a
common fixed point of a countable family of averaged operators
which allows for approximate computations of the operator values.

\begin{proposition}
\label{paQg507-04a}
For every $n\in\NN$, let $\alpha_n\in\zeroun$,
let $\lambda_n\in\left]0,1/\alpha_n\right[$, let $e_n\in\HH$,
and let $T_n\colon\HH\to\HH$ be an $\alpha_n$-averaged operator. 
Suppose that $S=\bigcap_{n\in\NN}\Fix T_n\neq\emp$ and that 
$\sum_{n\in\NN}\lambda_n\|e_n\|<\pinf$. Let 
$x_0\in\HH$ and set, for every $n\in\NN$,
\begin{equation}
\label{eaQg507-01a}
x_{n+1}=x_n+\lambda_n\big(T_nx_n+e_n-x_n\big).
\end{equation}
Then the following hold:
\begin{enumerate}
\item
\label{paQg507-04a-i}
Let $n\in\NN$, let $x\in S$, and set 
$\nu=\sum_{k\in\NN}\lambda_k\|e_k\|+2\sup_{k\in\NN}\|x_k-x\|$. 
Then $\nu<\pinf$ and 
\begin{align}
\|x_{n+1}-x\|^2
&\leq\|x_n+\lambda_n(T_nx_n-x_n)-x\|^2+\nu\lambda_n\|e_n\|
\label{fjskd43h8X4a}\\
&\leq\|x_n-x\|^2-\lambda_n(1/\alpha_n-\lambda_n)
\|T_nx_n-x_n\|^2+\nu\lambda_n\|e_n\|.
\label{eaQg507-04p}
\end{align}
\item
\label{paQg507-04a-ii}
$\sum_{n\in\NN}\lambda_n(1/\alpha_n-\lambda_n)
\|T_nx_n-x_n\|^2<\pinf$.
\item
\label{paQg507-04a-iii}
$(x_n)_{n\in\NN}$ converges weakly to a point in $S$ if and only if
every weak sequential cluster point of $(x_n)_{n\in\NN}$ is in $S$.
In this case, the convergence is strong if $\inte S\neq\emp$.
\item
\label{paQg507-04a-iv}
$(x_n)_{n\in\NN}$ converges strongly to a point in $S$ if and only if
$\varliminf d_S(x_n)=0$. 
\end{enumerate}
\end{proposition}
\begin{proof}
\ref{paQg507-04a-i}: 
Set 
\begin{equation}
\label{eaQg507-04u}
R_n=(1-1/\alpha_n)\Id+(1/\alpha_n)T_n
\quad\text{and}\quad\mu_n=\alpha_n\lambda_n. 
\end{equation}
Then $\Fix R_n=\Fix T_n$ and, by Proposition~\ref{p:av1}, $R_n$ 
is nonexpansive. Furthermore, \eqref{eaQg507-01a} can be 
written as
\begin{equation}
\label{eaQg507-04a}
x_{n+1}=x_n+\mu_n\big(R_nx_n-x_n\big)
+\lambda_ne_n,\quad\text{where}\quad\mu_n\in\zeroun. 
\end{equation}
Now set $z_n=x_n+\mu_n(R_nx_n-x_n)$. 
Since $x\in\Fix R_n$ and $R_n$ is nonexpansive, we have
\begin{align}
\label{eaQg507-04t}
\|z_n-x\|
&=\|(1-\mu_n)(x_n-x)+\mu_n(R_nx_n-R_nx)\|\nonumber\\
&\leq(1-\mu_n)\|x_n-x\|+\mu_n\|R_nx_n-R_nx\|\nonumber\\
&\leq\|x_n-x\|.
\end{align}
Hence, \eqref{eaQg507-04a} yields
\begin{align}
\|x_{n+1}-x\|
&\leq\|z_n-x\|+\lambda_n\|e_n\|
\label{eaQg507-04c}\\
&\leq\|x_n-x\|+\lambda_n\|e_n\|
\label{eaQg507-04b}
\end{align}
and, since $\sum_{k\in\NN}\lambda_k\|e_k\|<\pinf$, it follows from 
Lemma~\ref{l:7} that 
\begin{equation}
\label{e:mu}
\nu=\sum_{k\in\NN}\lambda_k\|e_k\|+
2 \underset{k\in\NN}{\rm\text{sup}\,}\|x_k-x\|<\pinf.
\end{equation}
Moreover, using \eqref{eaQg507-04c},
\eqref{eaQg507-04t}, and \cite[Corollary~2.14]{Livre1},
we can write
\begin{align}
\|x_{n+1}-x\|^2
&\leq\|z_n-x\|^2+(2\|z_n-x\|+\lambda_n\|e_n\|)
\lambda_n\|e_n\|\nonumber\\
&\leq\|z_n-x\|^2+(2\|x_n-x\|+\lambda_n\|e_n\|)
\lambda_n\|e_n\|\nonumber\\
&\leq\|(1-\mu_n)(x_n-x)+\mu_n(R_nx_n-x)\|^2+
\nu\lambda_n\|e_n\|\label{eaQg507-04s}\\
&=(1-\mu_n)\|x_n-x\|^2+\mu_n\|R_nx_n-x\|^2\nonumber\\
&\quad\;-\mu_n(1-\mu_n)\|R_nx_n-x_n\|^2+\nu\lambda_n\|e_n\|
\nonumber\\
&=(1-\mu_n)\|x_n-x\|^2+\mu_n\|R_nx_n-R_nx\|^2\nonumber\\
&\quad\;-\mu_n(1-\mu_n)\|R_nx_n-x_n\|^2+\nu\lambda_n\|e_n\|
\nonumber\\
&\leq\|x_n-x\|^2-\mu_n(1-\mu_n)\|R_nx_n-x_n\|^2+\nu\lambda_n\|e_n\|
\nonumber\\
&=\|x_n-x\|^2-\lambda_n(1/\alpha_n-\lambda_n)
\|T_nx_n-x_n\|^2+\nu\lambda_n\|e_n\|
\label{eaQg507-04i}\\
&\leq\|x_n-x\|^2+\nu\lambda_n\|e_n\|.
\label{eaQg507-04h}
\end{align}
Thus, \eqref{fjskd43h8X4a} follows from 
\eqref{eaQg507-04u} and \eqref{eaQg507-04s}, and 
\eqref{eaQg507-04i} provides \eqref{eaQg507-04p}.

\ref{paQg507-04a-ii}: This follows from 
\eqref{eaQg507-04p}, \eqref{e:mu}, and Lemma~\ref{l:7}.

\ref{paQg507-04a-iii}: 
The weak convergence statement follows from \eqref{e:mu}, 
\eqref{eaQg507-04h}, and \cite[Theorem~3.8]{Else01}, while 
the strong convergence statement follows from 
\cite[Proposition~3.10]{Else01}.

\ref{paQg507-04a-iv}:
By \cite[Corollary~4.15]{Livre1}, the sets $(\Fix T_n)_{n\in\NN}$ 
are closed, and so is therefore their intersection $S$. Hence, the 
result follows from \eqref{e:mu}, \eqref{eaQg507-04h}, 
\ref{paQg507-04a-ii}, and \cite[Theorem~3.11]{Else01}.
\end{proof}

The main result of this section is the following.

\begin{theorem}
\label{cr7seGhn3243gd4}
Let $\varepsilon\in\left]0,1/2\right[$, let $m\geq 2$ be an 
integer, let $x_0\in\HH$, and define $\phi$ as in 
\eqref{hhghwedw710f27}. 
For every $i\in\{1,\ldots,m\}$ and every $n\in\NN$, let 
$\alpha_{i,n}\in\zeroun$,
let $T_{i,n}\colon\HH\to\HH$ be $\alpha_{i,n}$-averaged, 
and let $e_{i,n}\in\HH$. For every $n\in\NN$, let
$\lambda_n\in\left]0,(1-\varepsilon)(1+\varepsilon
\phi(\alpha_{1,n},\ldots,\alpha_{m,n}))/
\phi(\alpha_{1,n},\ldots,\alpha_{m,n})\right]$ and set 
\begin{equation}
\label{e:main}
x_{n+1}=x_n+\lambda_n\bigg(T_{1,n}\bigg(T_{2,n}
\big(\cdots T_{m-1,n}(T_{m,n}x_n+e_{m,n})
+e_{m-1,n}\cdots\big)+e_{2,n}\bigg)
+e_{1,n}-x_n\bigg).
\end{equation}
Suppose that 
\begin{equation}
\label{eaQg507-04y}
S=\bigcap_{n\in\NN}\Fix(T_{1,n}\cdots T_{m,n})\neq\emp
\quad{and}\quad
(\forall i\in\{1,\ldots,m\})\quad
\sum_{n\in\NN}\lambda_n\|e_{i,n}\|<\pinf,
\end{equation}
and define
\begin{equation}
\label{eaQg507-04o}
(\forall i\in\{1,\ldots,m\})(\forall n\in\NN)\quad T_{i+,n}=
\begin{cases}
T_{i+1,n}\cdots T_{m,n},&\text{if}\;\;i\neq m;\\
\Id,&\text{if}\;\;i=m.
\end{cases}
\end{equation}
Then the following hold:
\begin{enumerate}
\item
\label{cr7seGhn3243gd4i}
$\sum_{n\in\NN}\lambda_n(1/\phi(\alpha_{1,n},\ldots,\alpha_{m,n})
-\lambda_n)\|T_{1,n}\cdots T_{m,n}x_n-x_n\|^2<\pinf$.
\item
\label{cr7seGhn3243gd4ii}
$(\forall x\in S)$ 
$\underset{1\leq i\leq m}{\text{\rm max}}\Sum_{n\in\NN}
\Frac{\lambda_n(1-\alpha_{i,n})}{\alpha_{i,n}}
\left\|(\Id-T_{i,n})T_{i+,n}x_{n}-(\Id-T_{i,n})T_{i+,n}x
\right\|^2<\pinf$.
\item
\label{cr7seGhn3243gd4iii}
$(x_n)_{n\in\NN}$ converges weakly to a point in $S$ if and only if
every weak sequential cluster point of $(x_n)_{n\in\NN}$ is in $S$.
In this case, the convergence is strong if $\inte S\neq\emp$.
\item
\label{cr7seGhn3243gd4iv}
$(x_n)_{n\in\NN}$ converges strongly to a point in $S$ if and only 
if $\varliminf d_S(x_n)=0$. 
\end{enumerate}
\end{theorem}
\begin{proof}
Let $n\in\NN$ and let $x\in S$. We can rewrite \eqref{e:main} as
an instance of \eqref{eaQg507-01a}, namely
\begin{equation}
\label{e:zn}
x_{n+1}=x_n+\lambda_n\big(T_nx_n+e_n-x_n\big),
\end{equation}
where
\begin{equation}
T_n=T_{1,n}\cdots T_{m,n}
\end{equation}
and
\begin{equation}
\label{e:5thonf}
e_n=T_{1,n}\bigg(T_{2,n}\big(\cdots 
T_{m-1,n}(T_{m,n}x_n+e_{m,n})+e_{m-1,n}\cdots\big)+e_{2,n}\bigg)
+e_{1,n}-T_{1,n}\cdots T_{m,n}x_n.
\end{equation}
It follows from 
Proposition~\ref{paQg507-03'} that 
\begin{equation}
\label{eaQg507-04x}
T_n\;\text{is $\alpha_n$-averaged, where}\;
\alpha_n=\phi(\alpha_{1,n},\ldots,\alpha_{m,n}).
\end{equation}
Since $\alpha_n\in\zeroun$, 
\begin{equation}
\label{e:w}
\frac{(1-\varepsilon)(1+\varepsilon\alpha_n)}{\alpha_n}
<\frac{(1-\varepsilon)(1+\varepsilon)}{\alpha_n}
=\frac{1-\varepsilon^2}{\alpha_n}
<\frac{1}{\alpha_n}
\end{equation}
and therefore $\lambda_n\in\left]0,1/\alpha_n\right[$, as
required in Proposition~\ref{paQg507-04a}.  

\ref{cr7seGhn3243gd4i}:
Using the nonexpansiveness of the operators 
$(T_{i,n})_{1\leq i\leq m}$, we derive from
\eqref{e:5thonf} that
\begin{align}
\label{e:qf8}
\|e_n\|
&\leq\|e_{1,n}\|+\nonumber\\[2mm]
&\quad\;\bigg\|T_{1,n}\bigg(T_{2,n}\big(\cdots T_{m-1,n}(T_{m,n}x_n
+e_{m,n})+e_{m-1,n}\cdots\big)+e_{2,n}\bigg)-T_{1,n}\cdots 
T_{m,n}x_n \bigg\|\nonumber\\[2mm]
&\leq\|e_{1,n}\|+\nonumber\\[2mm]
&\quad\;\bigg\|T_{2,n}\bigg(T_{3,n}\big(\cdots T_{m-1,n}
(T_{m,n}x_n+e_{m,n})+e_{m-1,n}\cdots\big)+e_{3,n}\bigg)+e_{2,n}-
T_{2,n}\cdots T_{m,n}x_n\bigg\|\nonumber\\[2mm]
&\leq\|e_{1,n}\|+\|e_{2,n}\|+\nonumber\\[2mm]
&\quad\;\bigg\|T_{3,n}\bigg(T_{4,n}\big(\cdots T_{m-1,n}
(T_{m,n}x_n+e_{m,n})+e_{m-1,n}\cdots\big)+e_{4,n}\bigg)+e_{3,n}-
T_{3,n}\cdots T_{m,n}x_n\bigg\|\nonumber\\[2mm]
&\;\;\vdots\nonumber\\
&\leq\sum_{i=1}^m\|e_{i,n}\|.
\end{align}
Accordingly, \eqref{eaQg507-04y} yields
\begin{equation}
\label{e:efwrf7}
\sum_{k\in\NN}\lambda_k\|e_k\|<\pinf.
\end{equation}
Hence, we deduce from 
Proposition~\ref{paQg507-04a}\ref{paQg507-04a-i} that
\begin{equation}
\label{e:mu'}
\nu=\sum_{k\in\NN}\lambda_k\|e_k\|+
2\underset{k\in\NN}{\rm\text{sup}\,}\|x_k-x\|<\pinf
\end{equation}
and from
Proposition~\ref{paQg507-04a}\ref{paQg507-04a-ii} that 
\begin{equation}
\label{e:2011-11-23}
\sum_{k\in\NN}\lambda_k\Big(\frac{1}{\alpha_k}-\lambda_k\Big)
\|T_kx_k-x_k\|^2<\pinf.
\end{equation}

\ref{cr7seGhn3243gd4ii}:
We derive from Proposition~\ref{p:av1} that
\begin{align}
\label{eaQg507-05a}
(\forall i\in\{1,\ldots,m\})(\forall (u,v)\in\HH^2)\nonumber\\
\|T_{i,n}u-T_{i,n}v\|^2
&\leq\|u-v\|^2-\frac{1-\alpha_{i,n}}{\alpha_{i,n}}
\|(\Id-T_{i,n})u-(\Id-T_{i,n})v\|^2.
\end{align}
Using this inequality $m$ times leads to 
\begin{align}
\label{eaQg507-04q}
\|T_nx_n-x\|^2
&=\left\|T_{1,n}\cdots T_{m,n}x_n-T_{1,n}\cdots T_{m,n}x\right\|^2
\nonumber\\
&\leq\|x_n-x\|^2-\sum_{i=1}^{m}\frac{1-\alpha_{i,n}}{\alpha_{i,n}}
\left\|(\Id-T_{i,n})T_{i+,n}x_n-(\Id-T_{i,n})T_{i+,n}x\right\|^2
\nonumber\\
&\leq\|x_n-x\|^2-\frac{\beta_n}{\lambda_n},
\end{align}
where
\begin{equation}
\label{Kiuxv532}
\beta_n=\lambda_n
\underset{1\leq i\leq m}{\text{\rm max}}
\bigg(\frac{1-\alpha_{i,n}}{\alpha_{i,n}}
\left\|(\Id-T_{i,n})T_{i+,n}x_n-(\Id-T_{i,n})T_{i+,n}x\right\|^2
\bigg).
\end{equation}
Note also that 
\begin{eqnarray}
\label{fjskd43h8X5a}
\lambda_n\leq\frac{(1-\varepsilon)(1+\varepsilon\alpha_n)}{\alpha_n}
&\Rightarrow&
\lambda_n\leq\frac{1+\varepsilon\alpha_n}
{(1+\varepsilon)\alpha_n}\nonumber\\
&\Leftrightarrow&
\bigg(1+\frac{1}{\varepsilon}\bigg)\lambda_n\leq
\frac{1}{\varepsilon\alpha_n}+1\nonumber\\
&\Leftrightarrow&
\lambda_n-1\leq\frac{1}{\varepsilon}
\bigg(\frac{1}{\alpha_n}-\lambda_n\bigg).
\end{eqnarray}
Thus, Proposition~\ref{paQg507-04a}\ref{paQg507-04a-i}, 
\eqref{e:zn}, and \cite[Corollary~2.14]{Livre1} yield
\begin{align}
\label{eaQg507-04r}
\|x_{n+1}-x\|^2
&\leq\|(1-\lambda_n)(x_n-x)+\lambda_n(T_nx_n-x)\|^2+
\nu\lambda_n\|e_n\|\nonumber\\
&=(1-\lambda_n)\|x_n-x\|^2+\lambda_n\|T_nx_n-x\|^2
+\lambda_n(\lambda_n-1)\|T_nx_n-x_n\|^2+\nu\lambda_n\|e_n\|
\nonumber\\
&\leq(1-\lambda_n)\|x_n-x\|^2+\lambda_n\|T_nx_n-x\|^2+\varepsilon_n,
\end{align}
where
\begin{equation}
\label{eaQg507-03i}
\varepsilon_n=\frac{\lambda_n}{\varepsilon}
\bigg(\frac{1}{\alpha_n}-\lambda_n\bigg)
\|T_nx_n-x_n\|^2+\nu\lambda_n\|e_n\|.
\end{equation}
On the one hand, it follows from \eqref{e:efwrf7}, 
\eqref{e:mu'}, and \eqref{e:2011-11-23} that
\begin{equation}
\label{eaQg507-05c}
\sum_{k\in\NN}\varepsilon_k<\pinf.
\end{equation}
On the other hand, combining \eqref{eaQg507-04q} and 
\eqref{eaQg507-04r}, we obtain
\begin{equation}
\label{eaQg507-05b}
\|x_{n+1}-x\|^2\leq\|x_n-x\|^2-\beta_n+\varepsilon_n.
\end{equation}
Consequently, Lemma~\ref{l:7} implies that
$\sum_{k\in\NN}\beta_k<\pinf$. 

\ref{cr7seGhn3243gd4iii}--\ref{cr7seGhn3243gd4iv}:
These follow from their counterparts in 
Proposition~\ref{paQg507-04a}.
\end{proof}

\begin{remark}
Theorem~\ref{cr7seGhn3243gd4} extends the results of 
\cite[Section~3]{Opti04}, where the relaxations parameters 
$(\lambda_n)_{n\in\NN}$ cannot exceed 1. Since these parameters
control the step-lengths of the algorithm, the proposed extension
can result in significant accelerations.
\end{remark}

\section{Application to forward-backward splitting}
\label{sec:4}

The forward-backward algorithm is one of the most versatile and
powerful algorithm for finding a zero of the sum of two maximally 
monotone operators (see \cite{Svva10,Opti14} and the references 
therein for historical background and recent developments). 
In \cite{Opti04},
the first author showed that the theory of averaged nonexpansive
operators provided a convenient setting for analyzing this
algorithm. In this section, we exploit the results of 
Sections~\ref{sec:2} and~\ref{sec:3} to further extend this 
analysis and obtain a new version of the forward-backward algorithm
with an extended relaxation range.

Let us recall a few facts about monotone set-valued operators
and convex analysis \cite{Livre1}. 
Let $A\colon\HH\to 2^{\HH}$ be a set-valued operator. 
The domain, the graph, and the set of zeros of $A$ are 
respectively defined by $\dom A=\menge{x\in\HH}{Ax\neq\emp}$,
$\gra A=\menge{(x,u)\in\HH\times\HH}{u\in Ax}$, and 
$\zer A=\menge{x\in\HH}{0\in Ax}$. The inverse of $A$ is
$A^{-1}\colon\HH\mapsto 2^{\HH}\colon u\mapsto 
\menge{x\in\HH}{u\in Ax}$, and the resolvent of $A$ is
\begin{equation}
\label{e:resolvent}
J_A=(\Id+A)^{-1}.
\end{equation}
This operator is firmly nonexpansive if $A$ is monotone, i.e.,
\begin{equation}
(\forall(x,y)\in\HH\times\HH)
(\forall(u,v)\in Ax\times Ay)\quad\scal{x-y}{u-v}\geq 0,
\end{equation}
and $\dom J_A=\HH$ if, furthermore, $A$ is maximally monotone,
i.e., there exists no monotone operator $B\colon\HH\to2^\HH$ such 
that $\gra A\subset\gra B$ and $A\neq B$. We denote by 
$\Gamma_0(\HH)$ the class of proper lower semicontinuous 
convex functions $f\colon\HH\to\RX$. Let
$f\in\Gamma_0(\HH)$. 
For every $x\in\HH$, $f+\|x-\cdot\|^2/2$ possesses a 
unique minimizer, which is denoted by $\prox_fx$. We have 
\begin{equation}
\label{e:prox2}
\prox_f=J_{\partial f},\quad\text{where}\quad
\partial f\colon\HH\to 2^{\HH}\colon x\mapsto
\menge{u\in\HH}{(\forall y\in\HH)\;\:\scal{y-x}{u}+f(x)\leq f(y)} 
\end{equation}
is the subdifferential of $f$. 

We start with a specialization of Theorem~\ref{cr7seGhn3243gd4} 
to $m=2$.

\begin{corollary}
\label{koGhn843gd4}
Let $\varepsilon\in\left]0,1/2\right[$ and let 
$x_0\in\HH$. For every every $n\in\NN$, let 
$\alpha_{1,n}\in\left]0,1/(1+\varepsilon)\right]$, let
$\alpha_{2,n}\in\left]0,1/(1+\varepsilon)\right]$, 
let $T_{1,n}\colon\HH\to\HH$ be $\alpha_{1,n}$-averaged, 
let $T_{2,n}\colon\HH\to\HH$ be $\alpha_{2,n}$-averaged, let
$e_{1,n}\in\HH$, and let $e_{2,n}\in\HH$. In addition, for every 
every $n\in\NN$, let
\begin{equation}
\label{ewfw7reww23a}
\lambda_n\in\left[\varepsilon,
\frac{(1-\varepsilon)(1+\varepsilon\phi_n)}{\phi_n}\right], 
\quad\text{where}\quad
\phi_n=\frac{\alpha_{1,n}+\alpha_{2,n}-2\alpha_{1,n}\alpha_{2,n}}
{1-\alpha_{1,n}\alpha_{2,n}}, 
\end{equation}
and set 
\begin{equation}
\label{ooio9o0vd59}
x_{n+1}=x_n+\lambda_n\Big(T_{1,n}\big(T_{2,n}x_n+
e_{2,n}\big)+e_{1,n}-x_n\Big).
\end{equation}
Suppose that 
\begin{equation}
\label{ooio9o0vd51}
S=\bigcap_{n\in\NN}\Fix(T_{1,n} T_{2,n})\neq\emp,\quad
\Sum_{n\in\NN}\lambda_n\|e_{1,n}\|<\pinf,
\quad\text{and}\quad
\Sum_{n\in\NN}\lambda_n\|e_{2,n}\|<\pinf.
\end{equation}
Then the following hold:
\begin{enumerate}
\item
\label{koGhn843gd4ii}
$(\forall x\in S)$ $\sum_{n\in\NN}
\|T_{1,n}T_{2,n}x_n-T_{2,n}x_n+T_{2,n}x-x\|^2<\pinf$.
\item
\label{koGhn843gd4ii+}
$(\forall x\in S)$
$\sum_{n\in\NN}\|T_{2,n}x_n-x_n-T_{2,n}x+x\|^2<\pinf$.
\item
\label{koGhn843gd4ii++}
$\sum_{n\in\NN}\|T_{1,n}T_{2,n}x_n-x_n\|^2<\pinf$.
\item
\label{koGhn843gd4iii}
Suppose that every weak sequential cluster point of 
$(x_n)_{n\in\NN}$ is in $S$. Then $(x_n)_{n\in\NN}$ converges 
weakly to a point in $S$, and the convergence is strong if 
$\inte S\neq\emp$.
\item
\label{koGhn843gd4iv}
Suppose that $\varliminf d_S(x_n)=0$. Then $(x_n)_{n\in\NN}$ 
converges strongly to a point in $S$.
\end{enumerate}
\end{corollary}
\begin{proof}
For every $n\in\NN$, since $\phi_n\in\zeroun$, 
$\varepsilon<1-\varepsilon<(1-\varepsilon)(1/\phi_n+\varepsilon)$
and $\lambda_n$ is therefore well defined in 
\eqref{ewfw7reww23a}. Overall, the present setting is encompassed 
by that of Theorem~\ref{cr7seGhn3243gd4} with $m=2$. 

\ref{koGhn843gd4ii}--\ref{koGhn843gd4ii+}:
Let $x\in S$. We derive from 
Theorem~\ref{cr7seGhn3243gd4}\ref{cr7seGhn3243gd4ii}
with $m=2$ that
\begin{equation}
\label{ewfw7reww23c}
\begin{cases}
\Sum_{n\in\NN}
\Frac{\lambda_n(1-\alpha_{1,n})}{\alpha_{1,n}}
\left\|(\Id-T_{1,n})T_{2,n}x_n-(\Id-T_{1,n})T_{2,n}x
\right\|^2<\pinf\\
\Sum_{n\in\NN}\Frac{\lambda_n(1-\alpha_{2,n})}{\alpha_{2,n}}
\left\|(\Id-T_{2,n})x_n-(\Id-T_{2,n})x
\right\|^2<\pinf.
\end{cases}
\end{equation}
However, it follows from the assumptions that 
\begin{equation}
\label{ewfw7reww23b}
(\forall n\in\NN)\quad T_{1,n}T_{2,n}x=x,\quad 
\Frac{\lambda_n(1-\alpha_{1,n})}{\alpha_{1,n}}\geq\varepsilon^2,
\quad\text{and}\quad 
\Frac{\lambda_n(1-\alpha_{2,n})}{\alpha_{2,n}}\geq\varepsilon^2.
\end{equation}
Combining \eqref{ewfw7reww23c} and \eqref{ewfw7reww23b} yields 
the claims.

\ref{koGhn843gd4ii++}: Let $x\in S$. Then, for every
$n\in\NN$,
\begin{align}
\label{ewfw7reww23d}
\|T_{1,n}T_{2,n}x_n-x_n\|^2
&=\|(T_{1,n}T_{2,n}x_n-T_{2,n}x_n+T_{2,n}x-x)
+(T_{2,n}x_n-x_n-T_{2,n}x+x)\|^2\nonumber\\
&\leq 2\|T_{1,n}T_{2,n}x_n-T_{2,n}x_n+T_{2,n}x-x\|^2
+2\|T_{2,n}x_n-x_n-T_{2,n}x+x\|^2.
\end{align}
Hence the claim follows from 
\ref{koGhn843gd4ii}--\ref{koGhn843gd4ii+}.

\ref{koGhn843gd4iii}--\ref{koGhn843gd4iv}: These follow
from  Theorem~\ref{cr7seGhn3243gd4}\ref{cr7seGhn3243gd4iii}%
--\ref{cr7seGhn3243gd4iv}.
\end{proof}

\begin{definition}{\rm\cite[Definition~2.3]{Sico10}}
\label{d:demir}
An operator $A\colon\HH\to 2^{\HH}$ is \emph{demiregular} at
$x\in\dom A$ if, for every sequence $((x_n,u_n))_{n\in\NN}$ in 
$\gra A$ and every $u\in Ax$ such that $x_n\weakly x$ and 
$u_n\to u$, we have $x_n\to x$.
\end{definition}

Here are some examples of demiregular monotone operators.

\begin{lemma}{\rm\cite[Proposition~2.4]{Sico10}}
\label{l:2009-09-20}
Let $A\colon\HH\to 2^{\HH}$ be monotone and suppose that 
$x\in\dom A$. Then $A$ is demiregular at $x$ in each of the 
following cases:
\begin{enumerate}
\item
\label{l:2009-09-20i}
$A$ is uniformly monotone at $x$, i.e., there exists 
an increasing function 
$\theta\colon\RP\to\RPX$ that vanishes only 
at $0$ such that
$(\forall u\in Ax)(\forall (y,v)\in\gra A)$
$\scal{x-y}{u-v}\geq\theta(\|x-y\|)$.
\item
\label{l:2009-09-20ii}
$A$ is strongly monotone, i.e., there exists $\alpha\in\RPP$ such
that $A-\alpha\Id$ is monotone.
\item
\label{l:2009-09-20iv-}
$J_A$ is compact, i.e., for every bounded set $C\subset\HH$,
the closure of $J_A(C)$ is compact. In particular, 
$\dom A$ is boundedly relatively compact, i.e., the intersection of 
its closure with every closed ball is compact.
\item
\label{l:2009-09-20vi}
$A\colon\HH\to\HH$ is single-valued with a single-valued continuous
inverse.
\item
\label{l:2009-09-20vii}
$A$ is single-valued on $\dom A$ and $\Id-A$ is demicompact, i.e., 
for every bounded sequence $(x_n)_{n\in\NN}$ in $\dom A$ such 
that $(Ax_n)_{n\in\NN}$ converges strongly, $(x_n)_{n\in\NN}$ 
admits a strong cluster point.
\item
\label{p:2009-09-20ii+}
$A=\partial f$, where $f\in\Gamma_0(\HH)$ is uniformly convex
at $x$, i.e., there exists an increasing function  
$\theta\colon\RP\to\RPX$ that vanishes only at $0$ such that 
\begin{equation}
(\forall\alpha\in\zeroun)(\forall y\in\dom f)\quad
f\big(\alpha x+(1-\alpha) y\big)+\alpha(1-\alpha)\theta(\|x-y\|)
\leq\alpha f(x)+(1-\alpha)f(y).
\end{equation}
\item
\label{p:2009-09-20ii++++}
$A=\partial f$, where $f\in\Gamma_0(\HH)$ and, for every
$\xi\in\RR$, $\menge{x\in\HH}{f(x)\leq\xi}$ is boundedly compact.
\end{enumerate}
\end{lemma}

Our extended forward-backward splitting scheme can now be
presented.

\begin{proposition}
\label{lI8uhT612}
Let $\beta\in\RPP$, let 
$\varepsilon\in\left]0,\min\{1/2,\beta\}\right[$, let $x_0\in\HH$,
let $A\colon\HH\to 2^{\HH}$ be maximally monotone, and let 
$B\colon\HH\to\HH$ be $\beta$-cocoercive, i.e.,
\begin{equation}
\label{e:cocoercive}
(\forall x\in\HH)(\forall y\in\HH)\quad
\scal{x-y}{Bx-By}\geq\beta\|Bx-By\|^2.
\end{equation}
Furthermore, let $(\gamma_n)_{n\in\NN}$ be a sequence in
$\left[\varepsilon,2\beta/(1+\varepsilon)\right]$, and
let $(a_n)_{n\in\NN}$ and $(b_n)_{n\in\NN}$ be sequences in $\HH$
such that $\sum_{n\in\NN}\|a_n\|<\pinf$ and 
$\sum_{n\in\NN}\|b_n\|<\pinf$. Suppose that $\zer(A+B)\neq\emp$
and, for every $n\in\NN$, let 
\begin{equation}
\label{ewfw7reww25}
\lambda_n\in\left[\varepsilon,(1-\varepsilon)
\bigg(2+\varepsilon-\Frac{\gamma_n}{2\beta}\bigg)\right]
\end{equation}
and set
\begin{equation}
\label{e:23juillet2007-1}
x_{n+1}=x_n+\lambda_n\Big(J_{\gamma_n A}
\big(x_n-\gamma_n(Bx_n+b_n)\big)+a_n-x_n\Big).
\end{equation}
Then the following hold:
\begin{enumerate}
\item
\label{lI8uhT612i}
$\sum_{n\in\NN}\|J_{\gamma_n A}(x_n-\gamma_nBx_n)-x_n\|^2<\pinf$.
\item
\label{lI8uhT612ii}
Let $x\in\zer(A+B)$. Then $\sum_{n\in\NN}\|Bx_n-Bx\|^2<\pinf$.
\item
\label{lI8uhT612iii}
$(x_n)_{n\in\NN}$ converges weakly to a point in $\zer(A+B)$.
\item
\label{lI8uhT612iv}
Suppose that one of the following is satisfied:
\begin{enumerate}
\item
\label{lI8uhT612iva}
$A$ is demiregular at every point in $\zer(A+B)$.
\item
\label{lI8uhT612ivb}
$B$ is demiregular at every point in $\zer(A+B)$.
\item
\label{lI8uhT612ivc}
$\inte S\neq\emp$.
\end{enumerate}
Then $(x_n)_{n\in\NN}$ converges strongly to a point 
in $\zer(A+B)$.
\end{enumerate}
\end{proposition}
\begin{proof}
We are going to establish the results as an application of
Corollary~\ref{koGhn843gd4}. Set 
\begin{equation}
(\forall n\in\NN)\;\;
\label{ooio9o0vd54}
T_{1,n}=J_{\gamma_n A},\quad T_{2,n}=\Id-\gamma_n B,\quad 
e_{1,n}=a_n,\quad\text{and}\quad e_{2,n}=-\gamma_nb_n.
\end{equation}
Then, for every $n\in\NN$, $T_{1,n}$ is $\alpha_{1,n}$-averaged
with $\alpha_{1,n}=1/2$ 
\cite[Remark~4.24(iii) and Corollary~23.8]{Livre1} and  
$T_{2,n}$ is $\alpha_{2,n}$-averaged with 
$\alpha_{2,n}=\gamma_n/(2\beta)$ \cite[Proposition~4.33]{Livre1}.
Moreover, for every $n\in\NN$,  
\begin{equation}
\label{ewfw7reww22s}
\phi_n=\frac{\alpha_{1,n}+\alpha_{2,n}-2\alpha_{1,n}\alpha_{2,n}}
{1-\alpha_{1,n}\alpha_{2,n}}=\frac{2\beta}{4\beta-\gamma_n} 
\end{equation}
and, therefore,
\begin{equation}
\label{ewfw7reww22t}
\lambda_n\in\left[\varepsilon,(1-\varepsilon)
(1+\varepsilon\phi_n)/\phi_n\right],
\end{equation}
in conformity with \eqref{ewfw7reww23a}. In turn,
Proposition~\ref{paQg507-03''}\ref{paQg507-03''i} yields
\begin{equation}
\label{ewfw7reww24z}
(\forall n\in\NN)\quad\lambda_n\leq\frac{1}{\phi_n}+\varepsilon
\leq\frac{1}{\alpha_{1,n}}+\varepsilon=2+\varepsilon.
\end{equation}
Consequently,
\begin{equation}
\label{ewfw7reww24t}
\sum_{n\in\NN}\lambda_n\|e_{1,n}\|=(2+\varepsilon)
\sum_{n\in\NN}\|a_n\|<\pinf\quad\text{and}\quad
\sum_{n\in\NN}\lambda_n\|e_{2,n}\|\leq
2(2+\varepsilon)\beta\sum_{n\in\NN}\|b_n\|<\pinf. 
\end{equation}
On the other hand, \cite[Proposition~25.1(iv)]{Livre1} yields
\begin{equation}
\label{e:fz1}
(\forall n\in\NN)\quad\zer(A+B)=\Fix(T_{1,n} T_{2,n}). 
\end{equation}
Altogether, $S=\zer(A+B)\neq\emp$, \eqref{ooio9o0vd51} is 
satisfied, and \eqref{e:23juillet2007-1} is an instance of 
\eqref{ooio9o0vd59}. 

\ref{lI8uhT612i}: This is a consequence of 
Corollary~\ref{koGhn843gd4}\ref{koGhn843gd4ii++} and
\eqref{ooio9o0vd54}.

\ref{lI8uhT612ii}: 
Corollary~\ref{koGhn843gd4}\ref{koGhn843gd4ii+} and
\eqref{ooio9o0vd54} yield 
\begin{align}
\sum_{n\in\NN}\|Bx_n-Bx\|^2
&=\sum_{n\in\NN}\gamma_n^{-2}\|T_{2,n}x_n-x_n-T_{2,n}x+x\|^2
\nonumber\\
&\leq\varepsilon^{-2}\sum_{n\in\NN}\|T_{2,n}x_n-x_n-T_{2,n}x+x\|^2
\nonumber\\
&<\pinf.
\end{align}

\ref{lI8uhT612iii}: Let $(k_n)_{n\in\NN}$ be a strictly
increasing sequence in $\NN$ and let $y\in\HH$ be such that 
$x_{k_n}\weakly y$. In view of 
Corollary~\ref{koGhn843gd4}\ref{koGhn843gd4iii}, it
remains to show that $y\in\zer(A+B)$. We set
\begin{equation}
\label{e:burn}
(\forall n\in\NN)\quad y_n=J_{\gamma_n A}(x_n-\gamma_nBx_n)\quad
\text{and}\quad u_n=\frac{x_n-y_n}{\gamma_n}-Bx_n,
\end{equation}
and note that
\begin{equation}
\label{e:jhUy74d}
(\forall n\in\NN)\quad u_n\in Ay_n.
\end{equation}
We derive from \ref{lI8uhT612i} that $y_n-x_n\to 0$, hence 
$y_{k_n}\weakly y$. Now let $x\in\zer(A+B)$. Then 
\ref{lI8uhT612ii} 
implies that $Bx_{n}\to Bx$, hence $u_n\to -Bx$. However, since 
\eqref{e:cocoercive} implies that $B$ is maximally monotone 
\cite[Example~20.28]{Livre1}, it follows from the properties 
$x_{k_n}\weakly y$ and $Bx_{k_n}\to Bx$ that $By=Bx$
\cite[Proposition~20.33(ii)]{Livre1}. Thus, $y_{k_n}\weakly y$
and $u_{k_n}\to -By$, and it therefore follows from 
\eqref{e:jhUy74d} and \cite[Proposition~20.33(ii)]{Livre1}
that $-By\in Ay$, i.e., $y\in\zer(A+B)$.

\ref{lI8uhT612iv}: By \ref{lI8uhT612iii}, there exists 
$x\in\zer(A+B)$ such that $x_n\weakly x$. In addition, we derive
from \eqref{e:burn}, \ref{lI8uhT612i}, and \ref{lI8uhT612ii}
that $y_n\weakly x$ and $u_n\to -B{x}\in A{x}$. 

\ref{lI8uhT612iva}:
Suppose that $A$ is demiregular at $x$. Then
\eqref{e:jhUy74d} yields $y_n\to{x}$ and 
\ref{lI8uhT612i} implies that $x_n\to{x}$.

\ref{lI8uhT612ivb}:
Suppose that $B$ is demiregular at $x$. Since $x_n\weakly x$ 
and $Bx_n\to Bx$ by \ref{lI8uhT612ii}, we have $x_n\to x$. 

\ref{lI8uhT612ivc}:
This follows from \ref{lI8uhT612iii} and 
Corollary~\ref{koGhn843gd4}\ref{koGhn843gd4iii}.
\end{proof}

\begin{remark}
Proposition~\ref{lI8uhT612} extends \cite[Corollary~6.5]{Opti04}
and \cite[Theorem~2.8]{Sico10}, which impose the additional
assumption that the relaxation parameters $(\lambda_n)_{n\in\NN}$
satisfy $(\forall n\in\NN)$ $\lambda_n\leq 1$. By contrast, the
relaxation range allowed in \eqref{ewfw7reww25} can be an 
arbitrarily large interval in $\left]0,2\right[$ and the maximum
relaxation is always strictly greater than 1.
\end{remark}

\begin{remark}
In Proposition~\ref{lI8uhT612}, the parameters
$(\gamma_n)_{n\in\NN}$ are allowed to vary at each iteration.
Now suppose that they are restricted to a fixed value
$\gamma\in\left]0,2\beta\right[$. Then, as in \eqref{e:zn},
\eqref{e:23juillet2007-1} reduces to
$x_{n+1}=x_n+\lambda_n(Tx_n+e_n-x_n)$, where 
$T=J_{\gamma A}(\Id-\gamma B)$ is $\alpha$-averaged and 
$e_n$ is given by \eqref{e:5thonf}. In this special
case, the weak convergence of $(x_n)_{n\in\NN}$ to a zero of $A+B$ 
can be derived from 
Proposition~\ref{paQg507-04a}\ref{paQg507-04a-iii} 
applied with $T_n\equiv T$, $\alpha_n\equiv\alpha$, and 
$(\lambda_n)_{n\in\NN}$ in $\left]0,1/\alpha\right[$ satisfying
$\sum_{n\in\NN}\lambda_n(1/\alpha-\lambda_n)=\pinf$ (see also 
\cite[Proposition~5.15(iii)]{Livre1}). This approach
was proposed in \cite[Theorem~25.8(i)]{Livre1} with the constant
$\alpha=\widetilde{\phi}(1/2,\gamma/(2\beta))=1/(1/2+
\text{\rm min}\{1,\beta/\gamma\})$ of \eqref{e:opti2004}, 
and revisited in 
\cite[Lemma~4.4]{Cond13} in the case of subdifferentials of convex
functions with the sharper constant
$\alpha={\phi}(1/2,\gamma/(2\beta))=2\beta/(4\beta-\gamma)$ of 
\cite[Theorem~3(b)]{Ogur02} (see Remark~\ref{KWghy5-f-09}).
\end{remark}

\begin{proposition}
\label{lI8uhT614}
Let $\beta\in\RPP$, let 
$\varepsilon\in\left]0,\min\{1/2,\beta\}\right[$, let $x_0\in\HH$, 
let $f\in\Gamma_0(\HH)$, let $g\colon\HH\to\RR$ be convex and 
differentiable with a $1/\beta$-Lipschitz gradient, and suppose 
that the set $S$ of solutions to the problem 
\begin{equation}
\label{ewfw7reww24u}
\minimize{x\in\HH}{f(x)+g(x)}
\end{equation}
is nonempty. Furthermore, let $(\gamma_n)_{n\in\NN}$ be a sequence 
in $\left[\varepsilon,2\beta/(1+\varepsilon)\right]$, and let 
$(a_n)_{n\in\NN}$ and $(b_n)_{n\in\NN}$ be sequences in $\HH$ such 
that $\sum_{n\in\NN}\|a_n\|<\pinf$ and 
$\sum_{n\in\NN}\|b_n\|<\pinf$. For every $n\in\NN$, let 
$\lambda_n\in\left[\varepsilon,%
(1-\varepsilon)(2+\varepsilon-\gamma_n/(2\beta))\right]$ and set
\begin{equation}
\label{ewfw7reww24v}
x_{n+1}=x_n+\lambda_n\Big(\prox_{\gamma_n f}
\big(x_n-\gamma_n(\nabla g(x_n)+b_n)\big)+a_n-x_n\Big).
\end{equation}
Then the following hold:
\begin{enumerate}
\item
\label{lI8uhT614i}
$\sum_{n\in\NN}
\|\prox_{\gamma_n f}(x_n-\gamma_n\nabla g(x_n))-x_n\|^2<\pinf$.
\item
\label{lI8uhT614ii}
Let $x\in S$. 
Then $\sum_{n\in\NN}\|\nabla g(x_n)-\nabla g(x)\|^2<\pinf$.
\item
\label{lI8uhT614iii}
$(x_n)_{n\in\NN}$ converges weakly to a point in $S$.
\item
\label{lI8uhT614iv}
Suppose that $\partial f$ or $\nabla g$ is demiregular at every 
point in $S$, or that $\inte S\neq\emp$. Then $(x_n)_{n\in\NN}$ 
converges strongly to a point in $S$.
\end{enumerate}
\end{proposition}
\begin{proof}
Using the same arguments as in \cite[Section~27.3]{Livre1}, one 
shows that this is the specialization of 
Proposition~\ref{lI8uhT612} to the case when $A=\partial f$ 
and $B=\nabla g$.
\end{proof}

\end{document}